\documentclass[a4paper]{amsart}
\usepackage{amssymb}
\usepackage{xcolor}
\usepackage{subfig}
\usepackage[hidelinks]{hyperref}
\usepackage{graphicx}
\usepackage{overpic}
\usepackage{breakcites}
\allowdisplaybreaks

\newcommand{\R}{\mathbb{R}}

\newcommand{\re}{\operatorname{Re}}
\newcommand{\im}{\operatorname{Im}}
\newcommand{\ext}[0]{\ensuremath{\mathrm{ext}}}
\numberwithin{equation}{section}
\numberwithin{figure}{section}

\theoremstyle{plain} 
\newtheorem{theorem}{Theorem}[section]
\newtheorem{lemma}[theorem]{Lemma}
\newtheorem{corollary}[theorem]{Corollary}
\newtheorem{proposition}[theorem]{Proposition}
\newtheorem{conjecture}[theorem]{Conjecture}
\theoremstyle{definition} 

\theoremstyle{remark}
\newtheorem{remark}[theorem]{Remark}

\def \tD {\mathrm{tD}}
\def \tDD {\mathrm{t\Delta}}
\def \tP {\mathrm{tP}}
\def \tPP {\mathrm{t\Pi}}
\def \oP {\mathrm{oP}}
\def \oD {\mathrm{oD}}
\def \oDD {\mathrm{o\Delta}}

\newcommand{\cM}{\mathcal{M}}
\newcommand{\cN}{\mathcal{N}}
\newcommand{\cD}{\mathcal{D}}
\newcommand{\cT}{\mathcal{T}}
\newcommand{\KK}{\bar{K}}
\newcommand{\EE}{\bar{E}}

\begin{document}

\title[A new $\oD$ family]{A new deformation family of Schwarz' D surface}

\author{Hao Chen}
\address[Chen]{Georg-August-Universit\"at G\"ottingen, Institut f\"ur Numerische und Angewandte Mathematik}
\email{h.chen@math.uni-goettingen.de}
\thanks{H.\ Chen is supported by Individual Research Grant from Deutsche Forschungsgemeinschaft within the project ``Defects in Triply Periodic Minimal Surfaces'', Projektnummer 398759432.}

\author{Matthias Weber}
\address[Weber]{Indiana University, Department of Mathematics}
\email{matweber@indiana.edu}


\date{\today}

\begin{abstract}
	We prove the existence of a new 2-parameter family $\oDD$ of embedded triply
	periodic minimal surfaces of genus 3.  The new surfaces share many properties
	with classical orthorhombic deformations of Schwarz' D surface, but also
	exotic in many ways.  In particular, they do not belong to Meeks'
	5-dimensional family.  Nevertheless, $\oDD$ meets classical deformations in a
	1-parameter family on its boundary.
\end{abstract}

\maketitle

\section{Introduction}

This is the first of two papers dealing with new 2-dimensional families of
embedded triply periodic minimal surfaces (TPMS) of genus three whose
1-dimensional ``intersections'' with the well-known Meeks family exhibit
singularities in the moduli space of TPMS.

In the past three decades, the classification of complete, embedded minimal
surfaces of finite topology in Euclidean space forms has largely been
accomplished for the smallest reasonable genus
(\cite{meeks1998,lazard-holly2001,perez2005,meeks2005,perez2007}).  In all
these cases, the moduli space of these surfaces has been found to be a smooth
manifold.

In the case of triply periodic minimal surfaces, no such classification has
been found. For the lowest possible genus 3, there is an explicit 5-dimensional
smooth family described by Meeks \cite{meeks1990} that contains most of the
then known examples, with the notable exception of Schwarz' H surfaces and
Schoen's Gyroid.  Work of Traizet~\cite{traizet2008} implies that the
H-surfaces belong to a second 5-dimensional family for which no explicit
description is known.  Our other paper will explore that family.

In this paper, we construct a new 2-dimensional family of embedded triply
periodic minimal surfaces of genus 3 that does not belong to the Meeks family
but whose closure meets the Meeks family in a 1-dimensional subset.  More
specifically, the surfaces in this subset are \emph{bifurcation instances} in
the sense that, with the same deformation of their lattices, they may deform
either within a classical 2-parameter Meeks family, or into a new 2-parameter
non-Meeks family.  Existence of the latter is the focus of this paper.

In fact, all these surfaces can be seen as orthorhombic deformations of
Schwarz' D surface.  Hence we begin with a description of the classical
orthorhombic deformations of D, which all belong to the Meeks family.

\medskip

\begin{figure}[hbt]
	\includegraphics[width=0.95\textwidth]{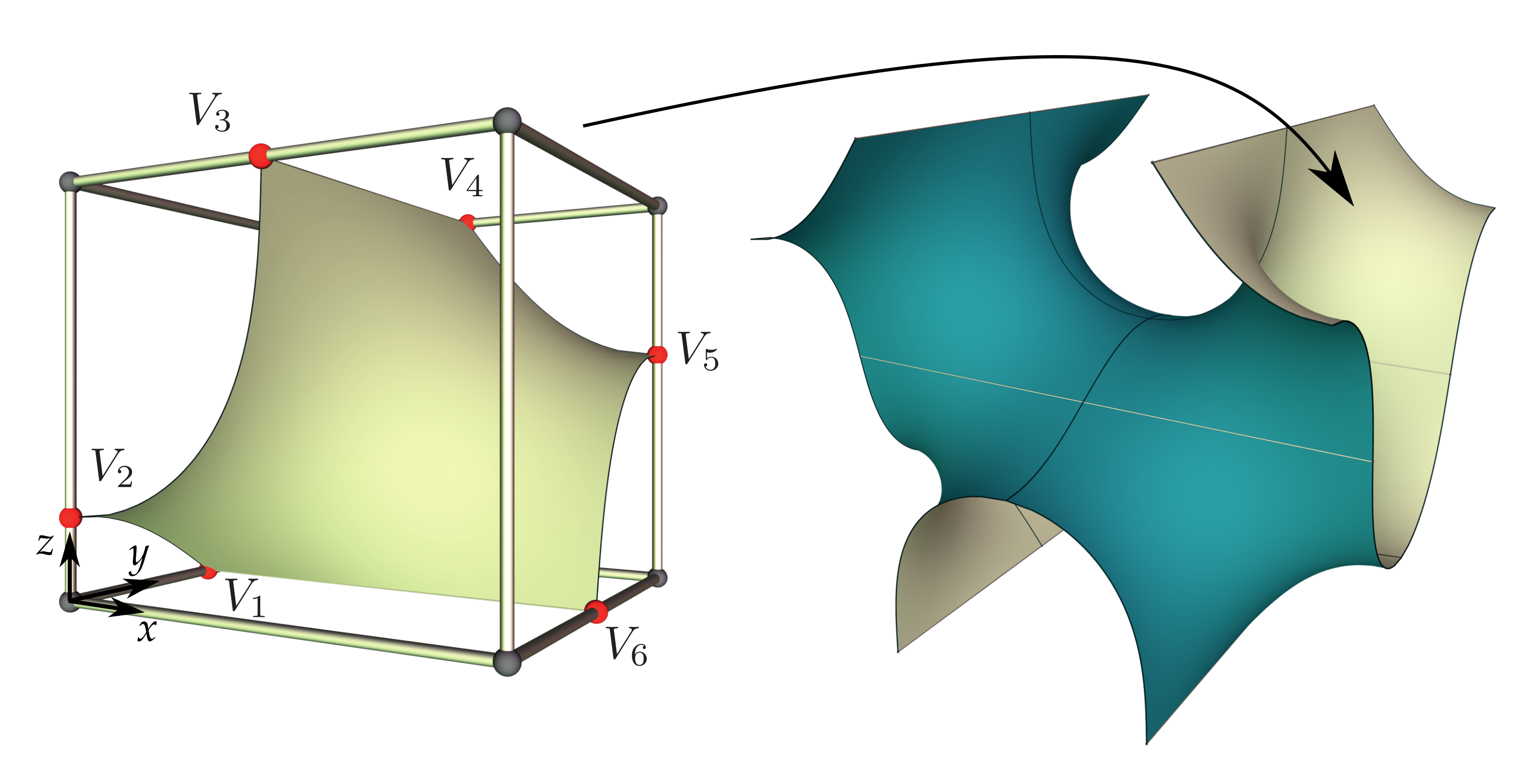}
	\caption{Fundamental Piece and Translational Fundamental Piece}
	\label{fig:fundpieces}
\end{figure}

Consider an embedded minimal surface $S$ inside an axis parallel box
$[-A,A]\times[-B,B]\times[0,1]$ that solves the following partially free
boundary problem: $S$ satisfies free boundary condition on the vertical planes
$x=\pm A$ and $y=\pm B$, and fixed (Plateau) boundary condition on the
horizontal segments $\{(x,0,0) \mid -A\le x \le A\}$ and $\{(0,y,1) \mid -B \le
y \le B\}$, and the intersection of $\partial S$ with each face of the box has
at most one component.  $S$ is therefore a right-angled embedded minimal
hexagon.  See Figure \ref{fig:fundpieces} (left) for an example.

Because the two horizontal segments are in the middle of the top and bottom
faces of the box, rotations about them and reflections in the lateral faces of
the box extend $S$ to an embedded TPMS $\tilde \Sigma$.  More specifically,
$\tilde \Sigma$ is invariant under the lattice $\Lambda$ spanned by $(4A,0,0)$,
$(0,4B,0)$ and $(2A,2B,2)$.  In the 3-torus $\R^3 / \Lambda$, $\Sigma = \tilde
\Sigma / \Lambda$ is a compact surface of genus $3$.  In
Figure~\ref{fig:fundpieces} (right) we show a translational fundamental domain
of $\tilde \Sigma$ nicely presented in a box.  It consists of eight copies of
$S$.

\begin{remark}
	For crystallographers, the orthorhombic lattice spanned by $(4A,0,0)$,
	$(0,4B,0)$ and $(0,0,4)$ is probably more convenient.  This is responsible
	for the letter ``o'' in our naming.  The quotient of $\tilde\Sigma$ by this
	lattice is a double cover of $\Sigma$, hence of genus $5$. 
\end{remark}

We use $\cD$ to denote the set of all TPMS obtained in this way.

\begin{figure}[hbt]
	\includegraphics[width=0.5\textwidth]{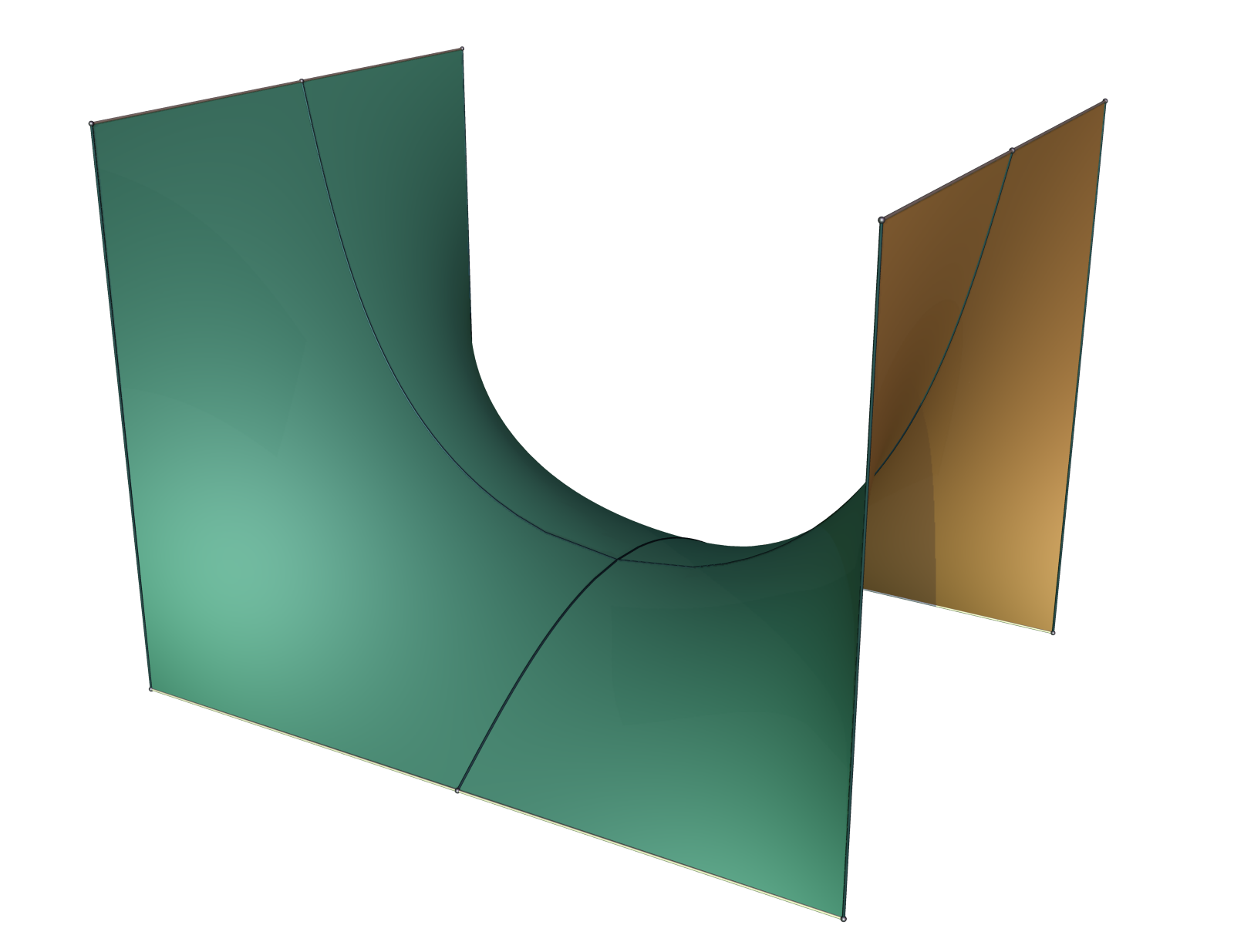}
	\caption{Plateau construction of $\oD$ surfaces}
	\label{fig:plateauD}
\end{figure}

A well-known family of surfaces in $\cD$ is the $\tD$ family of H.\ A.\
Schwarz, which is a tetragonal deformation family of his famous D surface.
They are obtained as described above with $A = B$ and $S$ containing the
vertical segment $\{(0,0,z) \mid 0 \le z \le 1\}$.  The same construction also
applies when $A \ne B$, yielding an orthorhombic deformation of Schwarz' D
surface, known as $\oD b$ in the literature to distinguish from another
orthorhombic deformation family $\oD a$; see~\cite{fischer1989,fogden1992}.  In
this paper, we simply use $\oD$ in place of $\oD b$.

An alternative (better known) construction of an $\oD$ surface starts with a
box of the same dimensions and then solves the Plateau problem for a polygonal
contour running along edges of the box, as shown in Figure~\ref{fig:plateauD}.
The Plateau solution is unique and therefore shares the symmetries of the
contour.  In particular, it has reflectional symmetries by vertical planes.  To
relate with the previous construction, just divide the minimal surface into
quarts by cutting along these planes, then extend one of the quarts by rotating
it about its vertical edge.

\medskip

The main result of this paper is to confirm the existence of another
2-parameter family in $\cD$.

\begin{theorem}\label{thm:main}
	There exists a second 2-parameter continuous family $\oDD$ in $\cD$, lacking
	the vertical straight line of the $\oD$ surfaces.
\end{theorem}

The $\oDD$ family, well hidden in the radiance of the famous $\oD$ family, is
understandably unexpected.  The second author confesses his complete bafflement
and initial disbelief when the first author provided him with evidence of
$\oDD$.  In Figure \ref{fig:Compare} we compare $\oD$ and $\oDD$ surfaces with
the same lattice (the surfaces in this figure actually have tetragonal
lattices, hence belong to $\tD$ and $\tDD$ subfamilies that we will discuss in
Section~\ref{sec:tD}.)

\begin{figure}[ht!]
  \begin{center}
		\captionsetup[subfigure]{labelformat=empty}
    \subfloat{
      \includegraphics[width=0.4\textwidth]{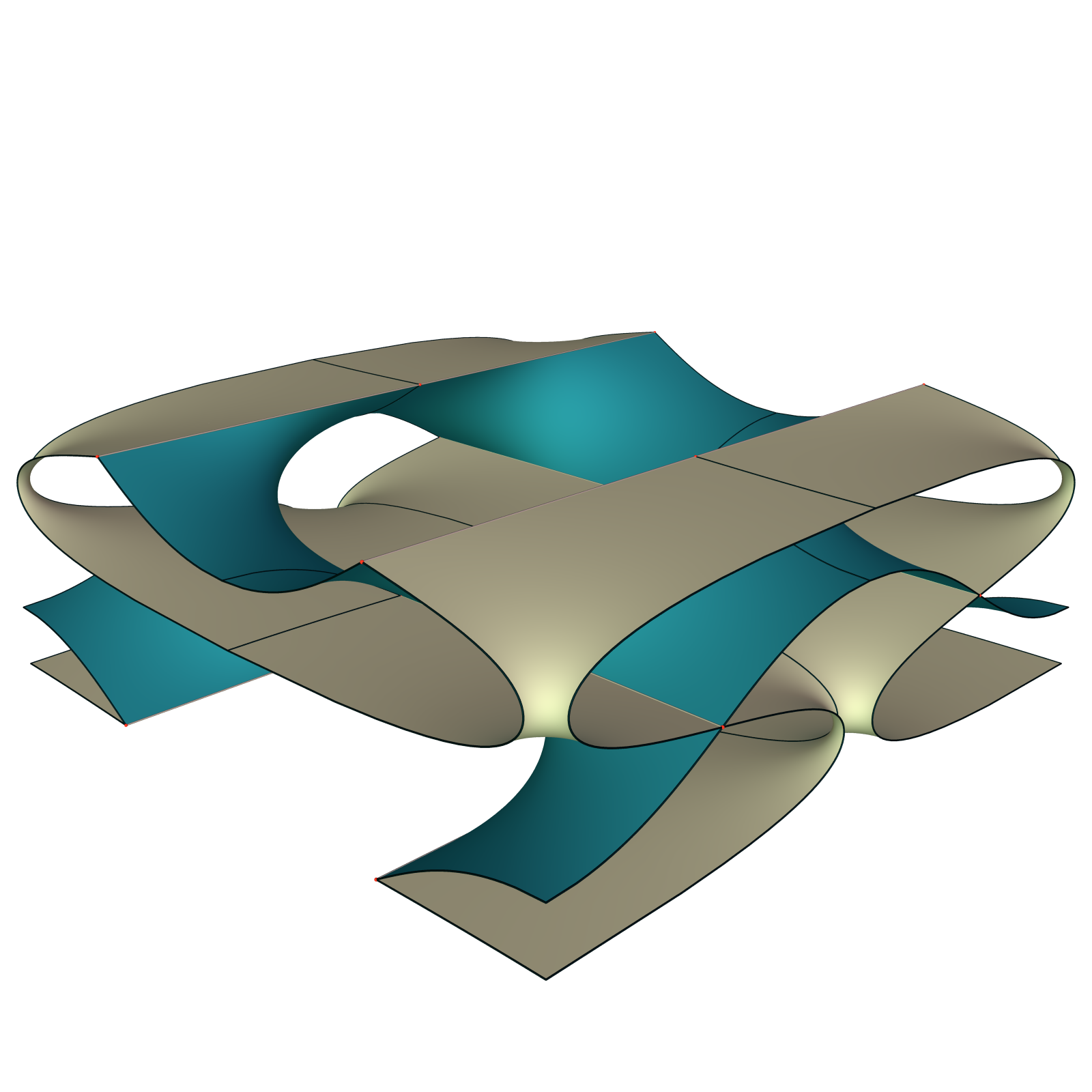}
    }
    \subfloat{
      \includegraphics[width=0.4\textwidth]{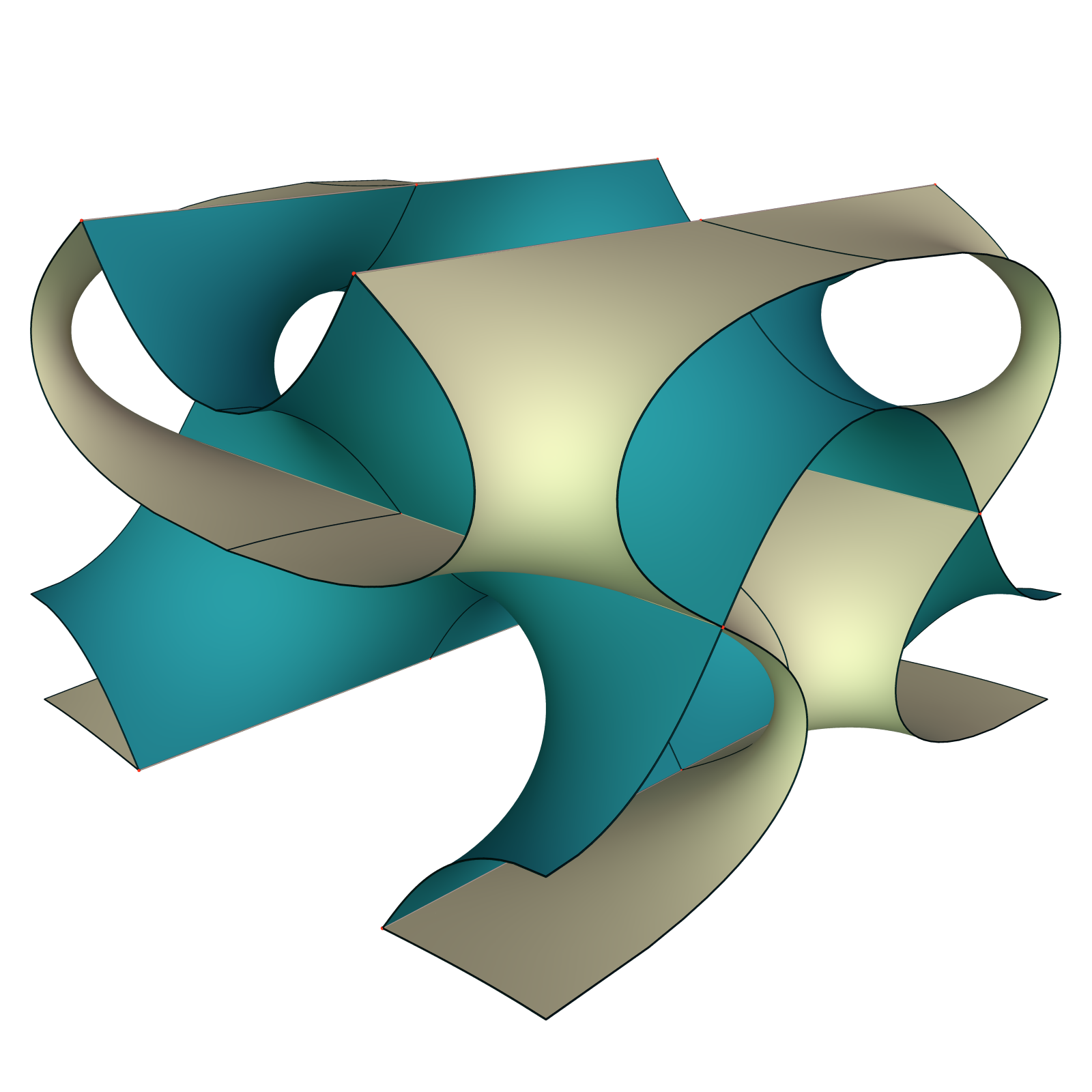}
    }\\
    \subfloat{
      \includegraphics[width=0.4\textwidth]{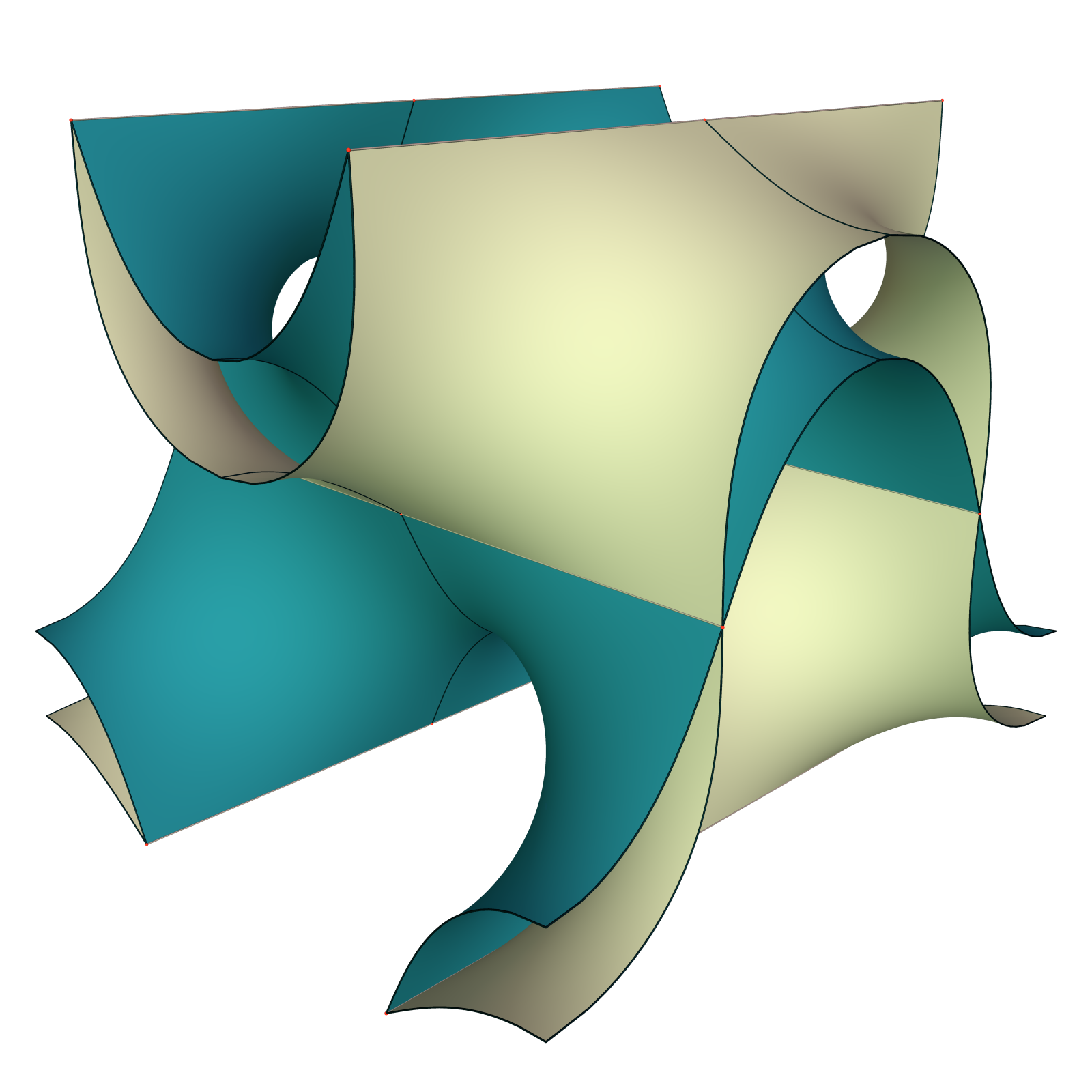}
    }
    \subfloat{
      \includegraphics[width=0.4\textwidth]{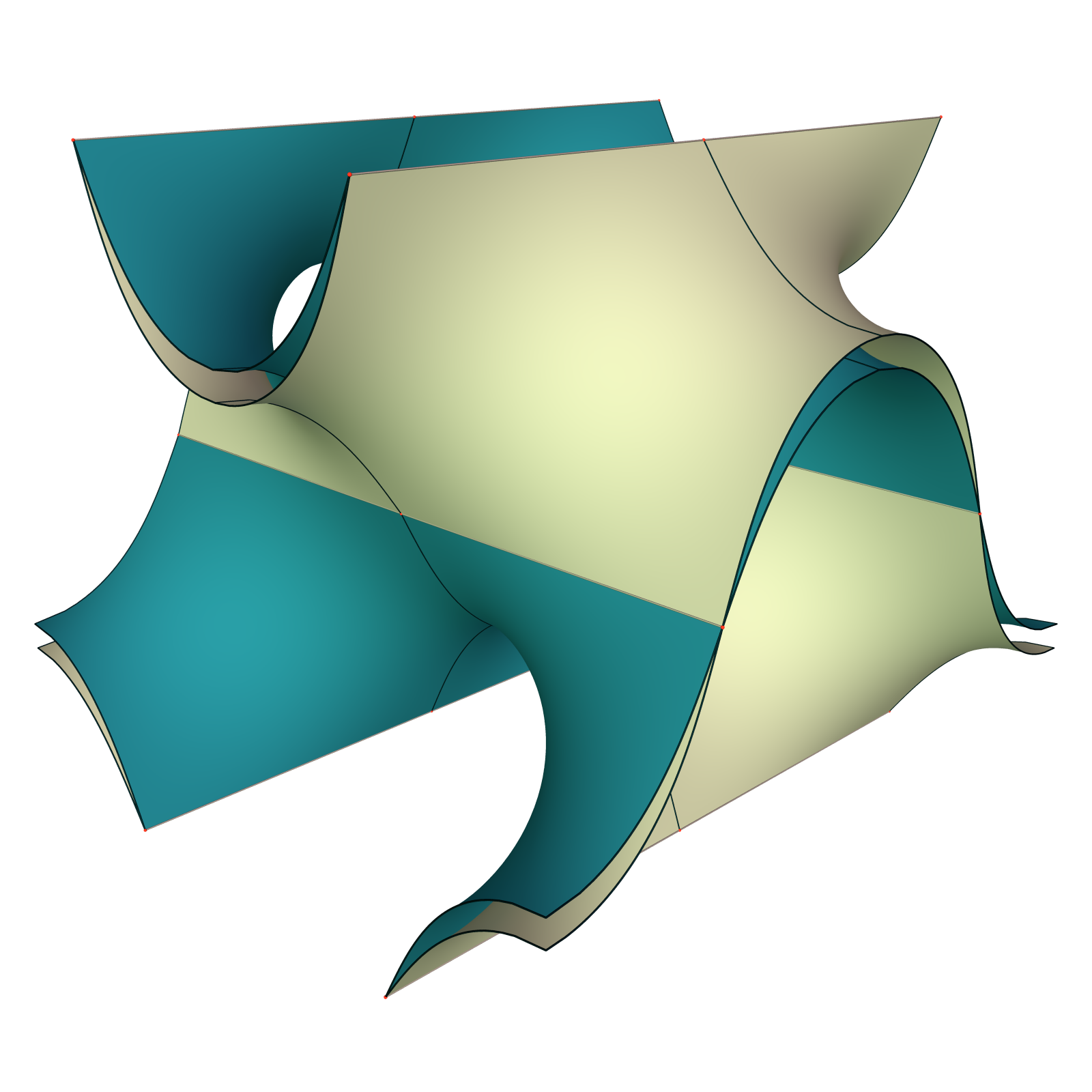}
    }
  \end{center}
  \caption{
    Comparison of $\oD$ (dark) and $\oDD$ (bright) surfaces with the same
    lattice.  These surfaces actually belong to the $\tD$ and $\tDD$ subfamily.
  }
  \label{fig:Compare}
\end{figure}

The $\oDD$ family is not merely a surprise.  Its significance is revealed in
the following proposition.

\begin{proposition}\label{prop:meeks}
	The surfaces in $\oDD$ do not belong to the Meeks family.  That is, the
	branched values of the Gauss map of an $\oDD$ surface do not form four
	antipodal pairs.  In fact, the only Meeks surfaces in $\cD$ are the $\oD$
	surface.  However, the closure $\overline\oDD$ intersects $\oD$ in a
	1-parameter family of TPMS.
\end{proposition}

\medskip

We now provide some context for the proposition.

For the purpose of this paper, a TPMS is a complete, embedded minimal surface
$\tilde \Sigma$ in Euclidean space $\R^3$ invariant under a lattice $\Lambda$
of Euclidean translations. The quotient $\Sigma = \tilde \Sigma /\Lambda$ then is
a compact Riemann surface in the 3-torus $\R^3 / \Lambda$.  The lowest possible
genus for a non-trivial TPMS is 3.  In this case, the Gauss map of $\Sigma$ has
degree 2, and the surface is therefore necessarily hyperelliptic.

The first examples of TMPS were given by H.\ A.\ Schwarz \cite{schwarz1890}
around 1867, with explicit Weierstrass data for very symmetric cases.  Schwarz
understood that the eight branched values of the Gauss map play a crucial role.
More generally, W.\ Meeks III \cite{meeks1990} explicitly constructed a family
$\cM$ of TPMS of genus 3.  He showed that if eight points on the sphere come in
four antipodal pairs, then they are the branched values of the Gauss map for
two conjugate TPMS of genus 3. The Meeks family $\cM$, considered up to
congruence and dilation, is a connected, smooth, (real) 5-dimensional manifold,
and includes almost all previously known examples.

Famous exceptions are the H surfaces of Schwarz, for which the branched values
are placed at the north pole, the south pole, and the vertices of a prism over
an equilateral triangle inside the sphere; and the Gyroid of A.\ Schoen
\cite{schoen1970}, whose Gauss map has the same branched values as those of
Schwarz' P and D surfaces, but does not belong to the Meeks family.  We use
$\cN$ to denote the complement of $\cM$ in the set of all TPMS of genus 3.
Since then, more examples in $\cN$ have been found, either as isolated examples
or as 1-parameter families, and some of them only numerically
\cite{fogden1993,fogden1999, weyhaupt2006,weyhaupt2008}.  Our 2-parameter
family $\oDD$ is therefore an important step towards the understanding of
non-Meeks TPMS of genus 3.

\medskip

Meeks' result is extended into the following rigidity statement:  In the
neighborhood of a non-degenerate TPMS, there is a bijection between TPMS and
lattices in $\R^3$; see \cite{koiso2014} for instance.  Hence up to congruence
and dilation, a non-degenerate TPMS belongs (locally) to a 5-parameter family.
Besides that, very little is known about the structure of $\cN$.  We would like
to conjecture that $\cN$ is, like $\cM$, connected and smooth, but none of
these is known.

There is evidence~\cite{fogden1993,fogden1999, weyhaupt2006, weyhaupt2008} that
$\cM$ and the closure $\overline{\cN}$ have non-empty intersection.
Proposition~\ref{prop:meeks} provides the first concrete example of such
intersection in the form of a 1-dimensional family of TPMS.  This is of
considerable importance for stability questions of TPMS.

A TPMS of genus 3 is called a \emph{bifurcation instance} if there are
non-congruent deformations (bifurcation branches) of the TPMS with the same
deformation of the lattice.  Koiso, Piccione and Shoda~\cite{koiso2014}
identified isolated \emph{bifurcation instances} among classical deformations
of TPMS; see also~\cite{ejiri2014,ejiri2018}.  They found bifurcation branches
for most of these bifurcation instances.  But for three ``exotic'' bifurcation
instances, they only suggested that a bifurcation branch from them would not be
a ``classical'' TPMS.

The intersection $\overline\oDD \cap \oD$ is a 1-parameter family of
bifurcation instances.  In particular, a 1-parameter subfamily of $\oDD$, which
we call $\tDD$, has the same tetragonal lattices as the $\tD$ family.  The
intersection $\overline\tDD \cap \tD$ contains a single TPMS, denoted by
$\tD^*$, which turns out to be one of the exotic bifurcation instances
in~\cite{koiso2014}.  We also find a bifurcation branch $\tPP$ from the
conjugate of $\tD^*$, another exotic bifurcation instance in Schwarz' $\tP$
family.  But $\tPP$ is not the focus of this paper, since it is nothing but a
classical $\oP a$ deformation~\cite{fischer1989,fogden1992}.

\medskip

For sufficiently large $A$ and $B$, the existence of $\oDD$ surfaces is implied
by results of Traizet~\cite{traizet2008}, who constructed TPMS by opening
catenoidal nodes among 2-tori.  The positions of the nodes have to satisfy a
balance condition, formulated in terms of elliptic functions, and a
non-degeneracy condition.  The Traizet limit of $\oDD$ was noted by the first
author in an earlier experimental work~\cite{chen2018}.  He used Brakke's
Surface Evolver~\cite{brakke1992} to numerically deform the TPMS from near the
Traizet limit up to Schwarz' $\tD$ family, and obtained the first images of
$\tDD$.  In particular, he observed that $\tDD$ eventually intersects $\tD$,
but Surface Evolver fails to converge near the intersection.  This failure can
now be explained by numeric bifurcation.

\medskip

Our paper is organized as follows: 

In Section \ref{sec:period}, we describe the Weierstrass data for surfaces in
$\cD$, prove their embeddedness, and formulate the period problem, depending on
three real positive parameters $a, b$ and $t$. The case $a=b$ corresponds to
the $\oD$ surfaces, where the period problem is automatically solved.  In the
case $a\ne b$, the period problem becomes 1-dimensional but is rather
complicated.

In Section \ref{sec:gauss} we show that, if $a \ne b$, the branched values of
the Gauss map can \emph{not} be antipodal.  This proves that $\cD \cap \cM =
\oD$, and that any solution with $a \ne b$ (namely $\oDD$) lies in $\cN$.

Section \ref{sec:exist} is dedicated to the existence proof of $\oDD$.  We show
that for any choice of $a \ne b$, there is a value of $t$ that solves the
period problem.  This is accomplished through a careful asymptotic analysis of
the period integrals.  We also conjecture the uniqueness of $t$ based on
numerical experiments.

To prove that $\oD \subset \cM$ and the closure of $\oDD \subset \cN$ have a
non-empty intersection, we consider in Section \ref{sec:elliptic} a modified
period problem that eliminates the trivial solutions coming from $\oD$. It
turns out that this period problem can be solved explicitly in terms of
elliptic integrals.

In section \ref{sec:tD} we consider the surfaces with tetragonal lattices.  They
are $\cD$ surfaces whose parameters satisfy $ab=t$. In this case, we obtain two
1-parameter families of surfaces: $\tD \subset \oD$ containing Schwarz' D
surface, and $\tDD \subset \oDD$.  The intersection $\overline \tDD \cap \tD$
contains a single TPMS $\tD^*$.  As the existence of $\tDD$ does not follow
from Section \ref{sec:exist}, we give an independent proof for this case using
an extremal length argument.

\subsection*{Acknowledgements}

The first author thanks his newborn daughter for keeping him awake through the
nights, which helped noticing the $\tDD$ family.

The second author thanks his teenage daughter for keeping him sleepless as
well, thus providing time to work on this paper.

We are grateful to the anonymous referee for suggestions and corrections after
carefully reading a previous version of the manuscript.

\section{Weierstrass Data and Period Problem}
\label{sec:period}

We parameterise a surface in $\cD$ with a Weierstrass representation defined on
the upper half plane such that the real axis is mapped to the boundary of the
hexagon $S$.  Let the vertices of $S$ be labeled by $V_1, V_2, \cdots, V_6$ as
in Figure~\ref{fig:fundpieces} (left).  Denote the preimage of $V_k$ by $v_k
\in \R$, and assume that $v_1<v_2<\ldots<v_6$. 

\begin{figure}[hbt]
	\includegraphics[width=0.8\textwidth]{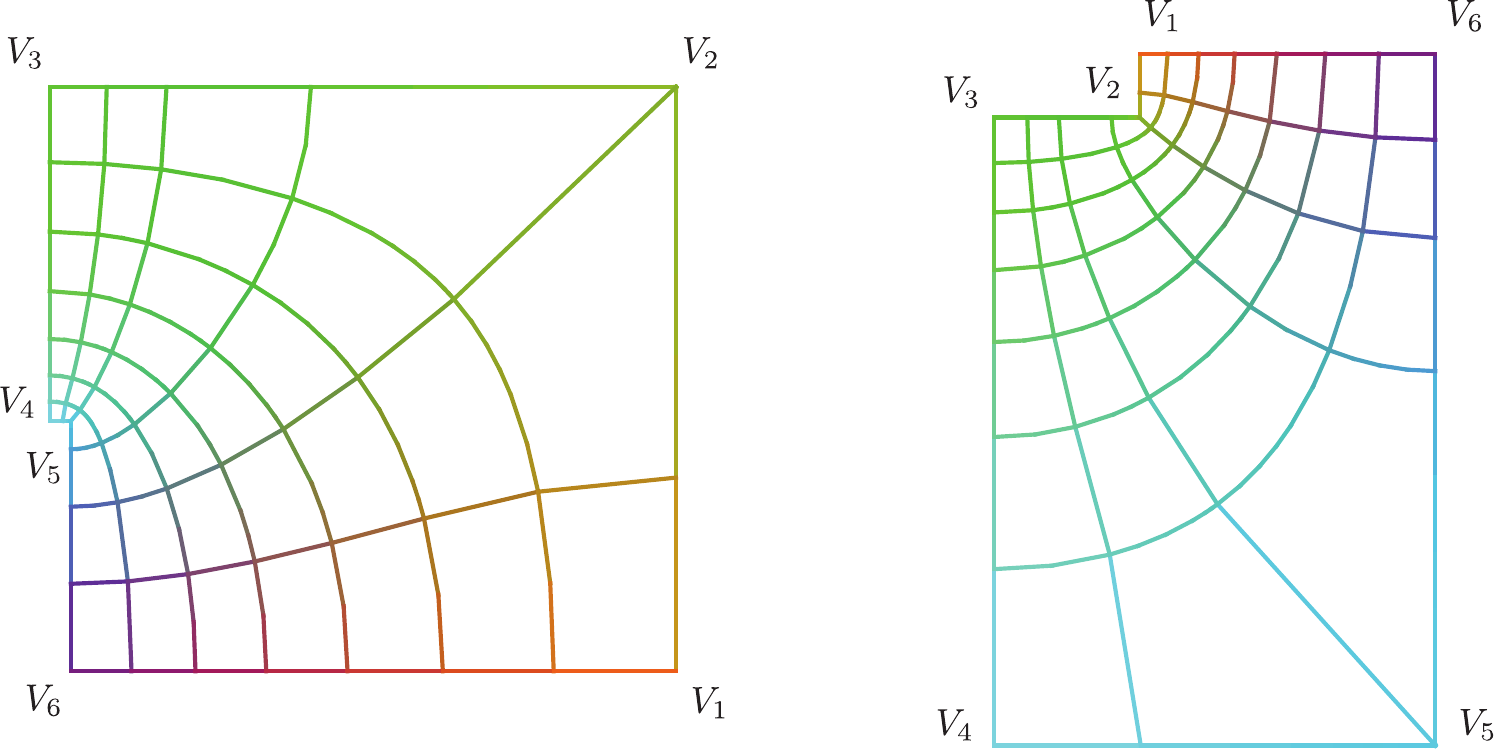}
	\caption{Images of a fundamental piece under $\Phi_1$ and $\Phi_2$.}
	\label{fig:phidomains}
\end{figure}

Given a $\cD$ surface, denote by $dh$ its height differential and by $G$ its
Gauss map. Let $\phi_1:=dh \cdot G$ and $\phi_2:=dh / G$.  The assumed boundary
symmetries of the surface imply that $\Phi_j: z \mapsto \int^z \phi_j$ ($j=1$
or $2$) map the upper half plane to ``right angled'' Euclidean hexagons.  The
interior angle is $270^\circ$ at $\Phi_1(v_5)$ and $\Phi_2(v_2)$.  Indeed, the
Gauss map is vertical at $V_2$ and $V_5$, hence $v_2$ and $v_5$ are
respectively the pole and the zero of $G$.  Interior angles at all other
vertices are $90^\circ$; see Figure \ref{fig:phidomains}.

Such maps are given by Schwarz-Christoffel maps.  More specifically, we have
\begin{align*} 
  \phi_1 :=   \rho\, & (z-v_1)^{-1/2} (z-v_2)^{-1/2} (z-v_3)^{-1/2} (z-v_4)^{-1/2} (z-v_5)^{+1/2} (z-v_6)^{-1/2}\, dz,\\
  \phi_2 := -\frac{1}{\rho} & (z-v_1)^{-1/2} (z-v_2)^{+1/2} (z-v_3)^{-1/2} (z-v_4)^{-1/2} (z-v_5)^{-1/2} (z-v_6)^{-1/2}\, dz,\\
  dh := -i\, & (z-v_1)^{-1/2} \makebox[0pt][l]{$\qquad\times$}\phantom{(z-v_2)^{-1/2}} (z-v_3)^{-1/2} (z-v_4)^{-1/2} \makebox[0pt][l]{$\qquad\times$}\phantom{(z-v_5)^{-1/2}} (z-v_6)^{-1/2}\, dz.
\end{align*}
Here, the real positive Lop{\'e}z-Ros factor $\rho$ determines scaling of the
image domains.  The Gauss map is $G:=i\rho(z-v_2)^{-1/2} (z-v_5)^{+1/2}$.

\begin{proposition}
	Up to congruence and dilation, the image of the upper half plane under the
	map
	\begin{equation} \label{eq:weierstrass}
		z \mapsto \re \int^z (\omega_1, \omega_2, \omega_3) = \re \int^z \left(
		\frac{1}{2}(\phi_2-\phi_1), \frac{i}{2}(\phi_2+\phi_1), dh \right)
	\end{equation}
	is \emph{almost} the fundamental hexagon of a $\cD$ surface in the following
	sense:  The intervals $v_1v_2$, $v_2v_3$, $v_4v_5$ and $v_5v_6$ are mapped to
	planar symmetry curves in the lateral faces of an axis parallel box.  The
	intervals $v_6v_1$ and $v_3v_4$ are mapped, respectively, to straight
	segments parallel to the $x$ and $y$ axis, but not necessarily in the middle,
	in the bottom and top faces of the box.
\end{proposition}

\begin{proof}
	Note that the integrand in $\phi_1$ (resp.\ $\phi_2$) is real positive
	(resp.\ negative) for $z>v_6$. This implies that the image of the segment
	$v_6v_1$ under the Schwarz-Christoffel map $\Phi_1$ (resp.\ $\Phi_2$) is
	horizontal rightward (resp.\ leftward), as in Figure \ref{fig:phidomains}.

	The Schwarz-Christoffel maps $\Phi_j$ and $z \mapsto \int^z dh$ can be
	continued by reflection across any edge to the lower half plane, inducing
	symmetries of the minimal surface.  We now determine what kind of symmetry is
	induced on each edge.
	
	For that, we only carry out a detailed analysis on the edge $v_6v_1$.  The
	integrands in both $\phi_j$ are real on $v_6v_1$, hence their continuations
	across this edge are given by $\overline{\phi_j(\overline{z})}$.  Meanwhile,
	the integrand in $dh$ is imaginary on $v_6v_1$, so its continuation is given
	by $-\overline{dh(\overline{z})}$.  Therefore, after crossing $v_6v_1$,
	$\re\omega_1$ remains unchanged while $\re\omega_2$ and $\re\omega_3$ change
	sign.  This means that the surface is extended by a rotation about a straight
	line parallel to the $x$-axis.
	
	Similar analysis on the other edges then prove that the image of the upper
	half plane under~\eqref{eq:weierstrass} has the claimed boundary curves.
	Note that the surface obtained is free of singularities.  Indeed, the metric
	is regular away from $v_k$, and the exponents at $v_k$ guarantee a smooth
	extension.
\end{proof}

\medskip

We now study the condition for the two horizontal segments to lie in the middle
of the top and the bottom faces of the box.  To this end, we introduce
notations for the edge lengths of the Euclidean hexagons 
\[
	I_k := \left| \int_{v_k}^{v_{k+1}} \phi_1 \right|, \qquad
	J_k := \left| \int_{v_k}^{v_{k+1}} \phi_2 \right|
\]
for $1 \le k \le 5$.  These are positive real numbers that depend analytically
on the parameters $v_1, \ldots, v_6$ and $\rho$.

\begin{proposition}
	The image of the upper half plane under the Weierstrass
	representation~\eqref{eq:weierstrass} is the fundamental hexagon of a surface
	in $\cD$ if and only if the following period conditions are satisfied:
	\begin{equation}
		\begin{aligned} \label{eq:periodcond}
			I_1+I_5 = J_1+J_5 \\
			I_2+I_4 = J_2+J_4
		\end{aligned}
	\end{equation}
\end{proposition}
\begin{proof}
	The bottom segment $V_6V_1$ lies in the middle of the bottom face if and only
	if
	\[
		\re \int_{v_1}^{v_2} \omega_2 = \re \int_{v_5}^{v_6} \omega_2 \ .
	\]
	This is equivalent to 
	\[
		\im \int_{v_1}^{v_2} (\phi_2+\phi_1) = \im \int_{v_5}^{v_6} (\phi_2+\phi_1) \ .
	\]
	Observe on $v_1v_2$ that the integrand in $\phi_1$ (resp.\ $\phi_2$) is
	positive (resp.\ negative) imaginary, and on $v_5v_6$ that the integrand in
	$\phi_1$ (resp.\ $\phi_2$) is negative (resp.\ positive) imaginary.  So the
	equation above can be written as
	\[
		I_1 - J_1 = J_5 - I_5,
	\]
	which proves the first period condition.  The second follows analogously.
\end{proof}

We can eliminate $\rho$ by taking the quotient of the two equations, therefore:

\begin{corollary} \label{cor:periodcondition}
	If 
	\[
		Q_I := \frac{I_1+I_5}{I_2+I_4} = \frac{J_1+J_5}{J_2+J_4} =: Q_J
	\]
	or, equivalently, if
	\begin{equation}\label{eq:quotient}
		Q := Q_I - Q_J = \frac{I_1+I_5}{I_2+I_4} - \frac{J_1+J_5}{J_2+J_4} = 0
	\end{equation}
	for some choice of $v_1,\ldots, v_6$, then $\rho \in \R_{>0}$ can be uniquely
	adjusted so that the period conditions~\eqref{eq:periodcond} are satisfied.
\end{corollary}

Thus we have expressed the period condition as a single equation $Q=0$, where
$Q$ depends on six parameters $v_1,\ldots, v_6$.  The number of parameters can
be reduced to three after a normalization by M\"obius transformations.  More
specifically, we can assume
\[
	v_1=-t, v_2 = -a, v_3=-1, v_4=1, v_5=b, v_6=t
\]
with $-t<-a<-1<1<b<t$.  We also assume that $a \le b$.  If it is not the case,
we may simply switch $a$ and $b$; this only exchanges $I_k$ and $J_{6-k}$, $1
\le k \le 5$, up to the scaling $\rho$, hence leaves $Q$ invariant.

\medskip

We note two special cases.

If $a=b$, the period conditions~\eqref{eq:periodcond} are satisfied
automatically with $\rho=1$.  In this case, the involution $z \mapsto
-\overline{z}$ induces an order-2 rotation of the surface about a vertical axis.
This can be seen by noting that $\omega_1$ and $\omega_2$ change sign but
$\omega_3$ keeps sign under this involution.  Indeed, on the imaginary axis
(fixed by the involution), $\phi_1$ and $\phi_2$ are conjugate and $dh$ is
real.  Hence the positive imaginary axis is mapped by the Weierstrass
representation~\eqref{eq:weierstrass} to the vertical straight segment between
the middle points of $V_3V_4$ and of $V_6V_1$, which serves as the axis of the
order-2 rotation.  This shows that the surface is in $\oD$.

If $ab=t$, we will see in Section~\ref{sec:tD} that the period conditions are
satisfied with $Q_I=Q_J=1$ and $\rho^4 = a/b$.  In this case, the involution $\iota: z \mapsto -t/z$
induces an order-2 orientation-preserving rotation of the surface around a horizontal axis, because
\[
	\iota^* dh = -dh \quad \text{and} \quad G(\iota(z)) G(z) = i.
\]
This rotation exchanges
$V_k$ with $V_{k+3}$, $1 \le k \le 3$.  In particular, the segments $V_6V_1$
and $V_3V_4$ must have the same length, implying that the bounding box has a
square base.  The unique fixed point of the involution, namely $i\sqrt{t}$, is
mapped to the fixed point of the rotation.  We will consider this case in
detail in Section \ref{sec:tD}.

\begin{proposition}
	The minimal hexagons $S$ in $\cD$ are embedded. Consequently, the triply
	periodic minimal surfaces generated by extending across symmetry lines are
	embedded as well.
\end{proposition}

\begin{proof}
	Denote the projection onto the $xz$-plane by $\varpi$, and let $V'_i =
	\varpi(V_i)$.  We will prove (refering to Figure \ref{fig:fundpieces} (left)):

	\begin{enumerate}
		\item The boundary of $S$ is a graph over a simple curve $\gamma$ in the
			$xz$-plane, except for the straight segment $V_3V_4$ which is parallel to
			the $y$-axis. Thus $\gamma$ bounds a simply connected (open) domain $\Omega$.

			To see this, note that the Gauss map $G:=i\rho(z-v_2)^{-1/2}
			(z-v_5)^{+1/2}$ is horizontal (i.e.\ perpendicular to the $y$-direction)
			along the segments $V_2V_3V_4V_5$ and strictly monotone.  This implies
			that the arcs $V_2'V_3'$ and $V_4'V_5'$ of $\gamma$ are simple, disjoint,
			and lie in the rectangle $[-A,A]\times[0,1]$. The remaining segments
			$V_5'V_6'$, $V_6'V_1'$ and $V_1'V_2'$ are straight segments on the
			boundary of that rectangle.

		\item The projection $\varpi(S)$ lies within $\overline{\Omega}$.
			
			To see this, assume the opposite.  Take a boundary point of $\varpi(S)$
			that does not lie in $\overline \Omega$. By (the contraposition of) the
			Implicit Function Theorem, its preimage on the $S$ has a horizontal
			normal (parallel to the $xz$-plane).  By the formula for the Gauss map,
			the only points with horizontal normal occur on the boundary of $S$, a
			contradiction.

		\item The projection $\varpi$ restricted to the interior of $S$ has the unique
			path and homotopy lifting properties.

			To see this, we again use that the interior of $S$ has no point with
			horizontal normal. The claim follows from the Implicit Function Theorem,
			applied in the compact region where the curve (or homotopy) resides.
	\end{enumerate}

	Then it follows that the interior of $S$ is a graph over $\Omega$: Otherwise,
	take a curve on $S$ that connects two distinct points in $\varpi^{-1}(p)$, $p
	\in \Omega$. Its projection onto $\Omega$ is closed in $\Omega$ and can be
	retracted onto $p$ within a compact subset of $\Omega$.  By the unique
	homotopy lifting property, the endpoints of the lifted curves stay the same,
	contradicting the assumption that they are two distinct points in
	$\varpi^{-1}(p)$. 
\end{proof}

\section{Branched Values of the Gauss Map}
\label{sec:gauss}

To locate the branched points of the Gauss map, we use the following simple
observation:

\begin{lemma}
	At every orthogonal intersection of a planar symmetry curve and a straight
	line on a minimal surface, the Gauss map has a branched point.
\end{lemma}

\begin{proof}
	At points on the straight line that are symmetric with respect to the
	symmetry plane, the Gauss map takes the same value.  Hence it cannot be
	single valued in a neighborhood of the intersection point.
\end{proof}

We now show 
\begin{theorem} 
	The branched values of the Gauss map of a surface in $\cD$ are antipodal if
	and only if $a=b$.
\end{theorem}

\begin{proof}
	By the Lemma, the Gauss map has branched points at $V_1$, $V_3$, $V_4$ and
	$V_6$.  On a translational fundamental domain, each of these points occurs
	twice, giving eight branched points as expected.

	Recall that the stereographically projected Gauss map is given by
	\[
		G(z) = i\rho(z+a)^{-1/2}(z-b)^{+1/2}.
	\]
	We then compute the branched values explicitly as
	\begin{align*}
		\pm G(+1) &= \mp\rho\sqrt\frac{b-1}{a+1},&
		\pm G(-1) &= \mp\rho\sqrt\frac{b+1}{a-1},\\
		\pm G(+t) &= \pm i\rho\sqrt\frac{t-b}{t+a},&
		\pm G(-t) &= \pm i\rho\sqrt\frac{t+b}{t-a}.\\
	\end{align*}
	Recall that $-t<-a<-1<1<b<t$, so the expressions under the square roots are
	all  positive real.  We then see that they lie on the real and imaginary axis,
	respectively, which helps matching them in possible antipodal pairs.  Recall
	that, after stereographic projection, the antipodal point of $z$ is
	$-1/\overline{z}$.  

	Assume that the branched values do occur in antipodal pairs and, for the sake
	of contradiction, that $a \ne b$.  First note that
	$G(+1)$ and $-G(+1)$ cannot be antipodal.  Otherwise, $G(-1)$ and $-G(-1)$
	must also be antipodal.  Then they must have the same norm, i.e.\
	$\frac{b-1}{a+1} = \frac{b+1}{a-1}$, forcing $a+b=0$ which violates our
	assumption.  Thus the only possibility is that $\pm G(-1)$ and $\pm G(+1)$
	are antipodal, with two possible choices of signs.  Either choice implies
	that
	\[
		\rho^4=\frac{a^2-1}{b^2-1}.
	\]
	The same analysis on $\pm G(\pm t)$ leads to
	\[
		\rho^4=\frac{t^2-a^2}{t^2-b^2}.
	\]
	Combining the two equations for $\rho^4$ shows, after a brief computation,
	that either $t=1$ or $a=b$.  The contradiction with our assumptions proves
	the ``only if''.

	For the ``if'' part, assume that $a=b$.  Then we find the branched points
	become antipodal (only) with $\rho=1$.  More specifically, we have
	\begin{align*}
		\pm G(+1) &= \mp\sqrt\frac{a-1}{a+1},&
		\pm G(-1) &= \mp\sqrt\frac{a+1}{a-1},\\
		\pm G(+t) &= \pm i\sqrt\frac{t-a}{t+a},&
		\pm G(-t) &= \pm i\sqrt\frac{t+a}{t-a}.\\
	\end{align*}
	Geometrically, these points on the unit sphere are vertices of two axis
	parallel rectangles in the planes $x=0$ and $y=0$, respectively.  Remarkably,
	an image of two such rectangles already appears in Figure~44 of the {\em
	Nachtrag} of Schwarz' paper ``Bestimmung einer speciellen Minimalfl\"ache"
	from 1867. 
\end{proof}

We note that in the case $a=b$ the branched values lie at the vertices of a
cube if and only of $a^2=b^2=t=3$.  This is the case of the classical $D$
surface of Schwarz.

\section{Existence of Non-Trivial Solutions}
\label{sec:exist}

Recall that $1 < a \le b < t$, and the periodic condition~\eqref{eq:quotient}
as we copy below
\[
	Q(a,b;t)= \frac{I_1+I_5}{I_2+I_4} - \frac{J_1+J_5}{J_2+J_4} = 0.
\]
The quantity $Q$ is our focus in the remaining of this paper.  From now on, we
will ignore the Lop\'ez-Ros factor $\rho$ in our calculations, since $Q$ is
independent of this factor.

We now prove the main theorem of this paper.
\begin{theorem} 
	If $a = b$, the period condition~\eqref{eq:quotient} is solved for any choice
	of $t$.

	If $a < b$, then there exists a value of $t$ that solves the period
	condition~\eqref{eq:quotient}.
\end{theorem} 

\begin{figure}
	\includegraphics[width=0.4\textwidth]{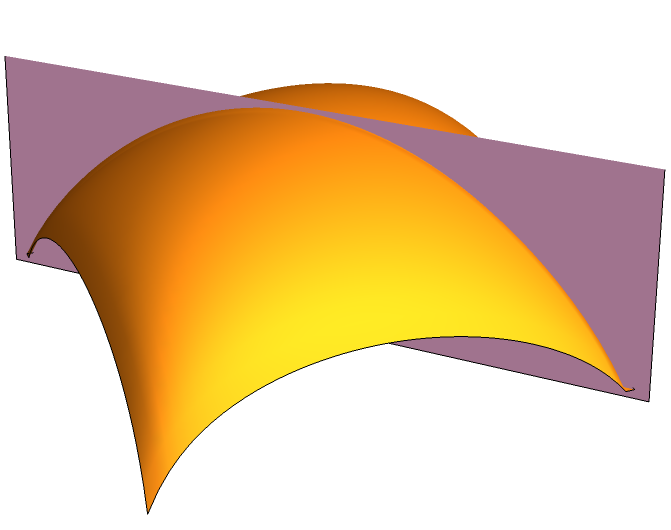}
	\caption{
		Numerical plot of the solution set of $Q(a,b;t) = 0$.  The vertical plane in gray is
		the set of trivial solutions with $a=b$ corresponding to the $\oD$ surfaces.
		The other surface is the set of non-trivial solutions corresponding to the
		$\oDD$ surfaces.
	} \label{fig:solutions}
\end{figure}

The solution set of $Q(a,b;t)=0$ is numerically plotted in
Figure~\ref{fig:solutions}.  The case $a = b$, shown here as a vertical plane
in gray, has been discussed in Section~\ref{sec:period}.  The case $a < b$, as
well as our main theorem, follows from the continuity of $Q$ in $t$, and the
following proposition.

\begin{proposition} \label{prop:lim1}
	If $1<a<b<t$ then
	\begin{align}
		\lim_{t \to b+} Q(a,b;t) &> 0,\label{eq:ttob}\\
		\lim_{t \to +\infty} Q(a,b;t) &= -\infty.\label{eq:ttoi}
	\end{align}
\end{proposition}

The remainder of this section is devoted to the proof of this proposition.

\medskip

We begin by analyzing the limit $t \to b+$.

\begin{proof}[Proof of~\eqref{eq:ttob}]
	We can evaluate the period integrals explicitly.  Recall that, if $p<q$, we have
	\[
		\int_{p}^{q} \sqrt\frac{1}{(q-z)(z-p)} \, dz = \pi,
		\qquad
		\int_{p}^{q} \sqrt\frac{z-p}{q-z} \, dz = \frac{q-p}{2}\pi.
	\]
	By the Mean Value Theorem for integrals, we have
	\[
		\lim_{t \to b+} I_1(a,b;t)
		= \lim_{t \to b+} \int_{-t}^{-a} \frac{1}{\sqrt{(t^2-z^2)(z^2-1)}}\sqrt\frac{b-z}{-a-z} \, dz
		= C \int_{a}^{b} \sqrt\frac{1}{(b-z)(z-a)} \, dz,
	\]
	where $C = 1/\sqrt{c^2-1}$ for some $c \in [a,t]$.  So this limit is finite
	and non-zero.  Similarly,
	\begin{gather*}
		\lim_{t \to b+} I_2(a,b;t) = C \int_{1}^{a} \sqrt\frac{1}{(a-z)(z-1)} \, dz,\\
		\lim_{t \to b+} J_1(a,b;t) = C \int_{a}^{b} \sqrt\frac{z-a}{b-z} \, dz,\qquad
		\lim_{t \to b+} J_2(a,b;t) = C \int_{1}^{a} \sqrt\frac{z-1}{a-z} \, dz,\\
		\lim_{t \to b+} I_4(a,b;t) = C \int_{1}^{b} \sqrt\frac{1}{z-1} \, dz,\qquad
		\lim_{t \to b+} J_5(a,b;t) = C \lim_{t \to b+} \int_{b}^{t} \sqrt\frac{1}{(t-z)(z-b)} \, dz
	\end{gather*}
	are all finite and non-zero.  Here $C$ denote any finite positive number.  On
	the other hand,
	\[
		\lim_{t \to b+} I_5(a,b;t) = C \lim_{t \to b+} \int_{b}^{t} \sqrt\frac{z-b}{t-z} \, dz = 0\\
	\]
	and
	\[
		\lim_{t \to b+} J_4(a,b;t) \ge C \lim_{t \to b+} \int_{1}^{b} \sqrt\frac{1}{(t-z)(b-z)} \, dz
	\]
	diverges to infinity.  Consequently, as $t \to b+$, $Q_I=(I_1+I_5)/(I_2+I_4)$
	has a finite and non-zero limit, while $Q_J=(J_1+J_5)/(J_2+J_4) \to 0$, hence
	$\lim_{t \to b+} Q > 0$.
\end{proof}

\medskip

Now we turn to the limit $t \to \infty$, which is more amusing.

\begin{proof}[Proof of~\eqref{eq:ttoi}]
	For the periods in the denominators, we note that
	\begin{align*} 
		\lim_{t\to \infty} t \cdot I_2(a,b;t) &= \int_{-a}^{-1} \sqrt\frac{1}{z^2-1} \sqrt\frac{b-z}{a+z}\, dz,\\
		\lim_{t\to \infty} t \cdot J_2(a,b;t) &= \int_{-a}^{-1} \sqrt\frac{1}{z^2-1} \sqrt\frac{a+z}{b-z}\, dz,\\
 		\lim_{t\to \infty} t \cdot I_4(a,b;t) &= \int_{1}^{b} \sqrt\frac{1}{z^2-1} \sqrt\frac{b-z}{a+z}\, dz,\\
 		\lim_{t\to \infty} t \cdot J_4(a,b;t) &= \int_{1}^{b}  \sqrt\frac{1}{z^2-1} \sqrt\frac{a+z}{b-z}\, dz
	\end{align*}
	are all finite.  We now show that
	\begin{equation} \label{eq:denominator}
		\lim_{t\to\infty} t \cdot (I_2+I_4) > \lim_{t\to\infty} t \cdot (J_2+J_4),
	\end{equation}
	or equivalently,
	\[
		\lim_{t\to\infty} t \cdot (I_2-J_2) > \lim_{t\to\infty} t \cdot (J_4-I_4).
	\]

	We prove this by considering the functions
	\begin{align*} 
		f(a,b) &= \lim_{t\to\infty} t \cdot (I_2-J_2) = \int_{1}^{a} \frac{2z-a+b}{\sqrt{(z^2-1)(a-z)(b+z)}}\, dz,\\
		g(a,b) &= \lim_{t\to\infty} t \cdot (J_4-I_4) = \int_{1}^{b} \frac{2z+a-b}{\sqrt{(z^2-1)(a+z)(b-z)}}\, dz,
	\end{align*}
	and show that $f(a,b)>g(a,b)$ for all $1<a<b$.  Note that $f(a,b)=g(b,a)$.
	Since 
	\[
		\frac{\partial}{\partial b} f(a,b) =  \int_1^a \frac{a+b}{\sqrt{(z^2-1)(a-z)(b+z)^3}}\, dz > 0,
	\]
	$f$ is monotone increasing in its second argument for $1<a<b$.  Then $g$ is
	monotone increasing in its first argument.  Note also that $f(a,a) = \pi$ is
	a constant.  Hence
	\[
		f(a,b)>f(a,a)= f(b,b) = g(b,b) > g(a,b),
	\]
	which finishes the proof of~\eqref{eq:denominator}.

	The periods in the numerators are more delicate to deal with, as they have
	logarithmic asymptotics.  For instance,
	\begin{align} 
 		t \cdot J_5(a,b;t)
 		&= \int_b^t \frac{t}{\sqrt{t^2-z^2}} \sqrt\frac{z+a}{z-b} \sqrt\frac{1}{z^2-1}\, dz\nonumber\\
  	&> \int_b^t  \frac{t}{\sqrt{t^2-z^2}} \frac{1}{z} \, dz\nonumber\\
  	&= \log \frac{\sqrt{t^2-b^2}+t}{b}, \label{eq:j5}
	\end{align}
	hence $t \cdot J_5(a,b;t)$ diverges to $+\infty$ as $t \to \infty$.

	Fortunately, the integrals $I_1$ and $J_1$ (and $I_5$ and $J_5$) have the
	same logarithmic singularities.  By the dominated convergence theorem, we
	obtain the following estimates:
	\begin{equation} \label{eq:numerator}
		\begin{aligned} 
			\lim_{t\to\infty} t \cdot (I_1-J_1)
			&= \lim_{t\to\infty} \int_{-t}^{-a} \frac{t(a+b)}{\sqrt{(t^2-z^2)(z^2-1)(b-z)(-a-z)}}\, dz\\
			&= \int_{-\infty}^{-a} \frac{a+b}{\sqrt{z^2-1}\sqrt{b-z}\sqrt{-a-z}}\, dz,\\
			\lim_{t\to\infty} t \cdot (I_5-J_5)
			&= \lim_{t\to\infty} \int_{b}^{t} \frac{-t(a+b)}{\sqrt{t^2-z^2}\sqrt{z^2-1}\sqrt{z-b}\sqrt{z+a}}\, dz\\
			&= \int_{b}^\infty \frac{-a-b}{\sqrt{z^2-1}\sqrt{z-b}\sqrt{z+a}}\, dz.
		\end{aligned}
	\end{equation}
	Note that they are finite and non-zero.

	Finally, we write
	\[
		Q(a,b;t) = \frac{t (I_1 - J_1) + t (I_5 - J_5)}{t I_2 + t I_4} + t (J_1 + J_5) \Big[\frac{1}{t I_2 + t I_4} - \frac{1}{t J_2 + t J_4} \Big].
	\]
	The part in the square bracket is negative by~\eqref{eq:denominator}.  As
	$t\to\infty$, the first fraction is bounded by~\eqref{eq:numerator}, and $J_5
	\to +\infty$.  This then concludes the proof of the proposition.
\end{proof}

Before ending this section, we propose the following uniqueness conjecture
based on numeric experiments:
\begin{conjecture}
	If $a < b$, then there exists a unique $t$ that solves the period
	condition~\eqref{eq:quotient}.
\end{conjecture} 

\section{Intersection with the Meeks-Locus}
\label{sec:elliptic}

By definition, the two families $\oD \subset \cM$ and $\oDD \subset \cN$ are
disjoint in $\cD$.  However, we will show in this section that the closure
$\overline\oDD$ intersects $\oD$, and give an explicit description of the
intersection in terms of elliptic integrals. This result is not strictly needed
for this paper, but gives insight into the nature of the bifurcation locus.

To make this precise, we use on $\cD$ the topology induced by the space of
possible Weierstrass data, which are determined by the four real parameters $a,
b, t$ and $\rho$.  Clearly, the convergence of Weierstrass data implies the
locally uniform convergence of the minimal surfaces.

The goal is to determine the intersection of the Meeks locus 
\[
	\oD = \{(a,b,t): Q(a,b;t) = 0, a=b, -t<-a<-1<1<b<t\}
\]
with the closure of the non-Meeks locus
\[
	\oDD = \{(a,b,t): Q(a,b;t) = 0, a \ne b, -t<-a<-1<1<b<t\}.
\]

The idea is to divide the function $Q(a,b;t)$ by $b-a$ and take the limit for
$a\to b$ to eliminate solutions in the Meeks locus. We claim: 

\begin{theorem} \label{thm:elliptic}
	The intersection $\overline\oDD \cap \oD$ is described by the equation
	\begin{equation} \label{eq:intersection}
		\KK(m_1)E(m_2) + \EE(m_1)K(m_2) = \KK(m_1)K(m_2),
	\end{equation}
	where
	\begin{align*} 
  	K(m) &= \int_0^{\pi/2} \frac{1}{\sqrt{1-m \sin^2(\theta)}}\, d\theta,  \\
  	E(m) &= \int_0^{\pi/2} {\sqrt{1-m \sin^2(\theta)}}\, d\theta  \\
	\end{align*}
	are complete elliptic integrals of the first and the second kind,
	$\KK(m)=K(1-m)$ and $\EE(m)=E(1-m)$ are the associated elliptic integrals,
	and the moduli
	\[
  	m_1 =\frac{a^2-1}{t^2-1}, \qquad m_2 =\frac{t^2}{a^2}\frac{a^2-1}{t^2-1}.
	\]
	Note that $0<m_1<m_2<1$.
\end{theorem}

\begin{remark}
	It is interesting to notice the similarity of~\eqref{eq:intersection} with
	the Legendre relation $\KK(m)E(m)+\EE(m)K(m)-\KK(m)K(m)=\pi/2$.
\end{remark}

Before we sketch the technical proof, we note that the function $Q$ can be
extended to a holomorphic function of its arguments $a,b$ and $t$ for $a$ near
$b$. To see this, note that the integrand of each of the integrals $I_k$ and
$J_k$ used in the definition of $Q$ can be adjusted by multiplication with a
constant factor $e^{i t}$ so that the absolute values are not necessary. The
square roots of the integrands cause a potential multivaluedness when the roots
$-t,-a,b$ and $t$ are close to each other, which is not the case for $a$ near
$b$. As $Q(a,a,t)=0$, this implies that also $\tilde Q$ extends to a
holomorphic function of its arguments. In particular, the extension of $\tilde
Q$ for real arguments is real analytic.

The theorem follows from the following proposition:

\begin{proposition}
	The function
	\[
		\tilde Q(a,b;t) = \frac{1}{b-a}Q(a,b;t) 
	\]
	extends analytically to $a=b$ by
	\[
		\tilde Q(a,a;t)= \frac{a(t^2-1)}{(a^2-1)(t^2-a^2)}\frac{\KK(m_1)K(m_2)-\KK(m_1)E(m_2)-\EE(m_1)K(m_2)}{K(m_2)^2}.
	\]
\end{proposition}

\begin{remark}
	Technical details in the following proof are omitted.  The integrals we need
	can all be evaluated in terms of the complete elliptic integrals of the first
	and the second kind.  Integral tables in~\cite{byrd1971} have been very
	helpful for this purpose, especially after a well-known computer algebra
	system failed us here.
\end{remark}

\begin{proof}
	With the help of the integral tables in~\cite{byrd1971}, we obtain the
	following explicit evaluation of the periods.
	\begin{align*}
  	(I_1 + I_5)(a,a;t) = (J_1 + J_5)(a,a;t) &= \frac{2 \KK(m_1)}{\sqrt{t^2-1}},\\
  	(I_2 + I_4)(a,a;t) = (J_2 + J_4)(a,a;t) &= \frac{2 K(m_2)}{\sqrt{t^2-1}}.
	\end{align*}

	Then we evaluate the derivatives
	\[
		I'_k(a,a;t) = \frac{\partial}{\partial b} \Bigr|_{a=b} I_k(a,b,t),\qquad
		J'_k(a,a;t) = \frac{\partial}{\partial b} \Bigr|_{a=b} J_k(a,b,t),
	\]
	and obtain
	\begin{align*} 
		(I'_2+I'_4)(a,a;t) &= \frac{K(m_2)}{a\sqrt{t^2-1}},\qquad
		(I'_1+I'_5)(a,a;t)  = 0,\\
		(J'_1+J'_5)(a,a;t) &= \phantom{\frac{K(m_2)}{a\sqrt{t^2-1}}-} \frac{2a\KK(m_1)}{\sqrt{t^2-1}(t^2-a^2)}-\frac{2a\EE(m_1)\sqrt{t^2-1}}{(a^2-1)(t^2-a^2)},\\
		(J'_2+J'_4)(a,a;t) &= \frac{K(m_2)}{a\sqrt{t^2-1}} - \frac{2 a K(m_2)}{\sqrt{t^2-1}(a^2-1)}+\frac{2a E(m_2)\sqrt{t^2-1}}{(a^2-1)(t^2-a^2)}.
	\end{align*}

	Finally, by L'H\^opital,

	\begin{multline}
 		\lim_{a \to b} \frac{1}{b-a}Q(a,b;t)
 		= \frac{\partial Q}{\partial b} \Bigr|_{a=b}\\
		= \frac{a(t^2-1)}{(a^2-1)(t^2-a^2)}\frac{\KK(m_1)K(m_2)-\KK(m_1)E(m_2)-\EE(m_1)K(m_2)}{K(m_2)^2}.
	\end{multline}

\end{proof}

\section{The Tetragonal Case}
\label{sec:tD}

We denote by $\cT$ surfaces in $\cD$ with tetragonal lattice.  That is, their
unit cells are prisms over squares.  We have seen that this occurs when $ab=t$.
Again, we have the classical family $\tD = \oD \cap \cT$ when $a=b=\sqrt{t}$.
The final specialization arises when $t=3$.  In this case, all diagonals and
midpoint bisectors of the embedded minimal hexagon are straight lines, and we
have the classical $D$ surface.

In this section we will show that $\tDD=\oDD \cap \cT$ is non-empty and, in
fact, contains a 1-parameter family of surfaces meeting $\tD$ on its boundary.
More specifically, these surfaces are characterized by $ab=t$, hence they all
admit a conformal involution, that exchanges $V_k$ with $V_{k+3}$, $1 \le k \le
3$.

\begin{lemma}\label{lem:tetraPerCond}
	When $ab=t$, the period condition is solved if and only if $I_1+I_5=I_2+I_4$,
	in which case $\rho^4=a/b$.
\end{lemma}
\begin{proof}
	The assumption $t=ab$ implies that
	\[
		I_k = \rho^2 \sqrt\frac{b}{a} J_{k+3}, \qquad \text{and} \qquad J_k =
		\frac{1}{\rho^2}\sqrt\frac{a}{b} I_{k+3}
	\]
	for $k=1,2,3$.  Therefore
	\[
		Q_I=\frac{I_1+I_5}{I_2+I_4}=\frac{J_2+J_4}{J_1+J_5}=Q_J^{-1}.
	\]
	Hence $Q=Q_I-Q_J=0$ implies that $Q_I=1$.
\end{proof}

We use this lemma to construct right angled hexagons that solve the period
problem.

Begin with an axis parallel rectangle $R$ of size $1 \times A$, where $1< A <
2$ is the height; see Figure~\ref{fig:extremal}.  Draw a line from the top left
vertex of $R$ in the $45^\circ$ south-east direction.  Choose a point $p$ on
this line in the lower half of $R$ (possible because $A<2$), and use it as the
bottom right vertex of a smaller rectangle $R'$ with the same symmetries.  Cut
the rectangular annulus between $R$ and $R'$ into four along the symmetry
lines.  The top right component is a right angled hexagon that solves the
period problem. 

Its conformal type, however, is still too general.  It needs to have a
holomorphic involution permuting the edges. 

\begin{figure}[hbt]
	\includegraphics[width=0.5\textwidth]{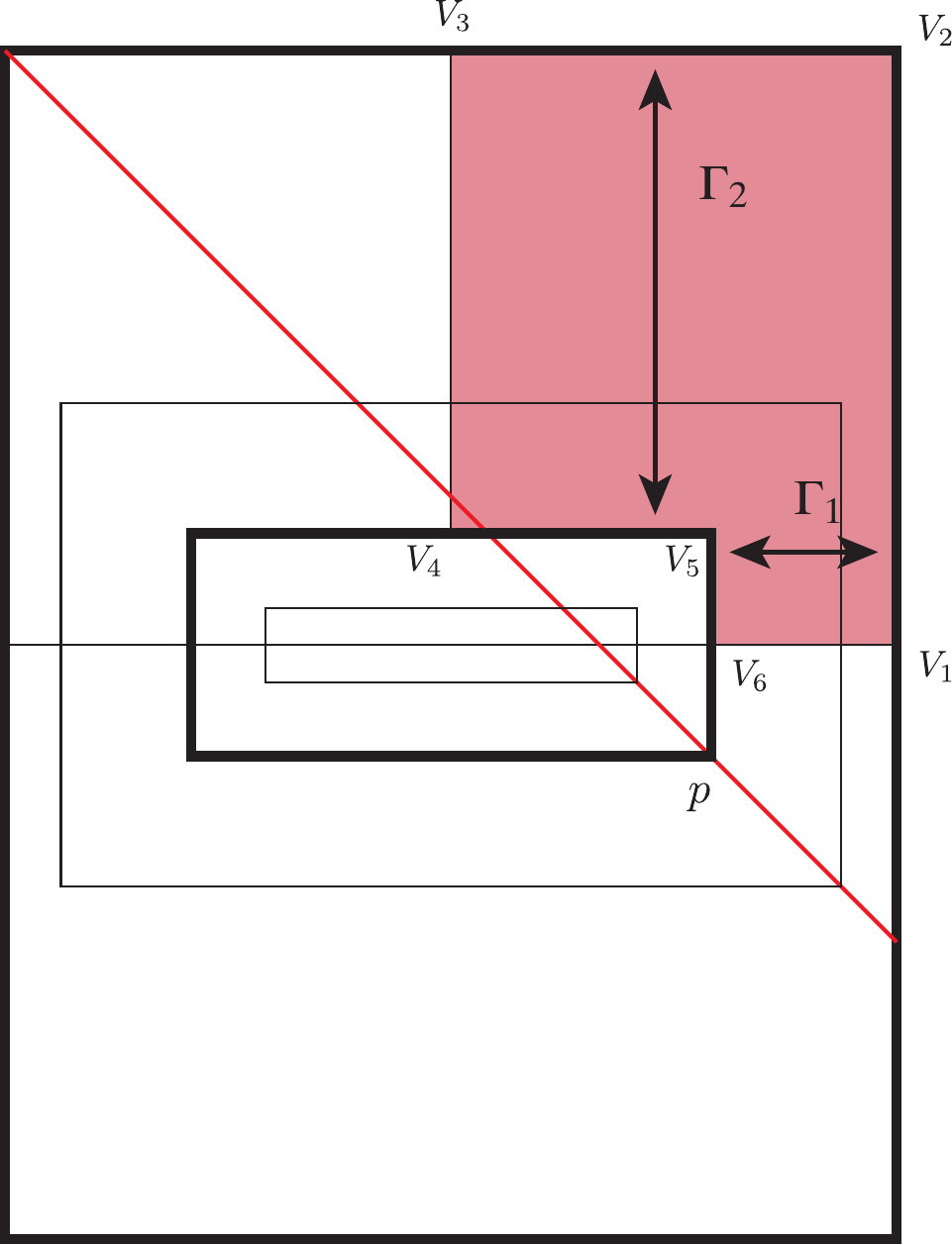}
	\caption{Existence Proof for $\tDD$}
	\label{fig:extremal}
\end{figure}

\begin{theorem} 
	For any choice of $1<A<2$, there is a choice of $p$ so that the hexagon has a
	conformal involution.
\end{theorem}

\begin{proof}
	The proof uses an extremal length argument.

	Consider the curve families $\Gamma_1$ and $\Gamma_2$ connecting edges as in
	Figure~\ref{fig:extremal}. These families are obtained from each other by the
	topological order 2 rotation.  So in a conformally correct hexagon, they need
	to have the same extremal length.

	Vice versa, we claim that if $\ext\Gamma_1=\ext\Gamma_2$ then the hexagon has a
	conformal involution. To see this, we map the hexagon to the upper half plane
	by the inverse of the Schwarz-Christoffel map $z \mapsto \int^z \phi_1$. The
	hexagon vertices $V_i$ are mapped to real numbers $v_i$, and the curve family
	$\Gamma_1$ is mapped to the curves family connecting the edge $v_1v_2$ with
	the edge $v_5v_6$.  Therefore its extremal length is that of the conformal
	rectangle $v_1v_2v_5v_6$, and thus determines the cross ratio of these four
	points.  Similarly, the extremal length of $\Gamma_2$ determines the cross
	ratio of the four points $v_2v_3v_4v_5$.  If we normalize the $v_i$ as
	before, the equality of these cross ratios
	\[
		\frac{(a+t)(b+t)}{2t(a+b)} = \frac{(a+1)(b+1)}{2(a+b)}
	\]
	implies that $ab=t$, so the hexagon has indeed a conformal involution.

	Thus we have to show that we can adjust the position of $p$ so that the two
	extremal lengths are equal.  Note that moving $p$ to the left will pinch the
	vertical edge $V_5V_6$, while moving $p$ to the right will pinch the
	horizontal edge $V_6V_1$. This shows that the extremal length of $\Gamma_1$
	will vary between infinity and 0.  On the other hand, during this variation,
	the extremal length of $\Gamma_2$ stays bounded away from 0 and infinity.
	Hence there must be a $p$ for which $\ext\Gamma_1 = \ext\Gamma_2$.
\end{proof}

Note that the $\tD$ family corresponds to the case when both rectangles
degenerate to squares.

\begin{remark}
	In the tetratonal case $ab=t$, the substitution $\zeta = z - t/z$ allows us
	to express the $I_k$'s in terms of the complete elliptic integral $K(\mu)$
	with complex modulus
	\[
		\mu = \frac{(1+a)(1-b)}{2}\frac{(\sqrt{t}-i)^2}{(t-1)}\frac{1}{(\sqrt{a}+i\sqrt{b})^2}.
	\]
	Then the period condition in Lemma~\ref{lem:tetraPerCond} is equivalent to
	\[
		\cot\Big(\arg\frac{K(\mu)+iK'(\mu)}{\sqrt{b}-i\sqrt{a}}\Big) = \frac{\sqrt{b}-\sqrt{a}}{\sqrt{b}+\sqrt{a}}.
	\]
\end{remark}

The intersection with $\tD$ can be determined explicitly using the equation
from Section \ref{sec:elliptic}.  Note that for $a=b=\sqrt{t}$, we have
\[
	m=m_2=1-m_1 = \frac{a^2}{1+a^2}.
\] 
Simplifying~\eqref{eq:intersection} shows that the intersection occurs when
\[
	2E(m)=K(m).
\]
This is solved numerically with $a=a^*\approx 2.17966$.  We use
$\tD^*$ to denote the surface with parameters $a=b=\sqrt{t}=a^*$.  In
Figure~\ref{fig:tD} we compare Schwarz' D surface, the most symmetric surface
in the $\tD$ family, with the surface $\tD^*$ at the junction of $\tD$ and
$\tDD$.

\begin{figure}[hbt]
	\includegraphics[width=0.48\textwidth]{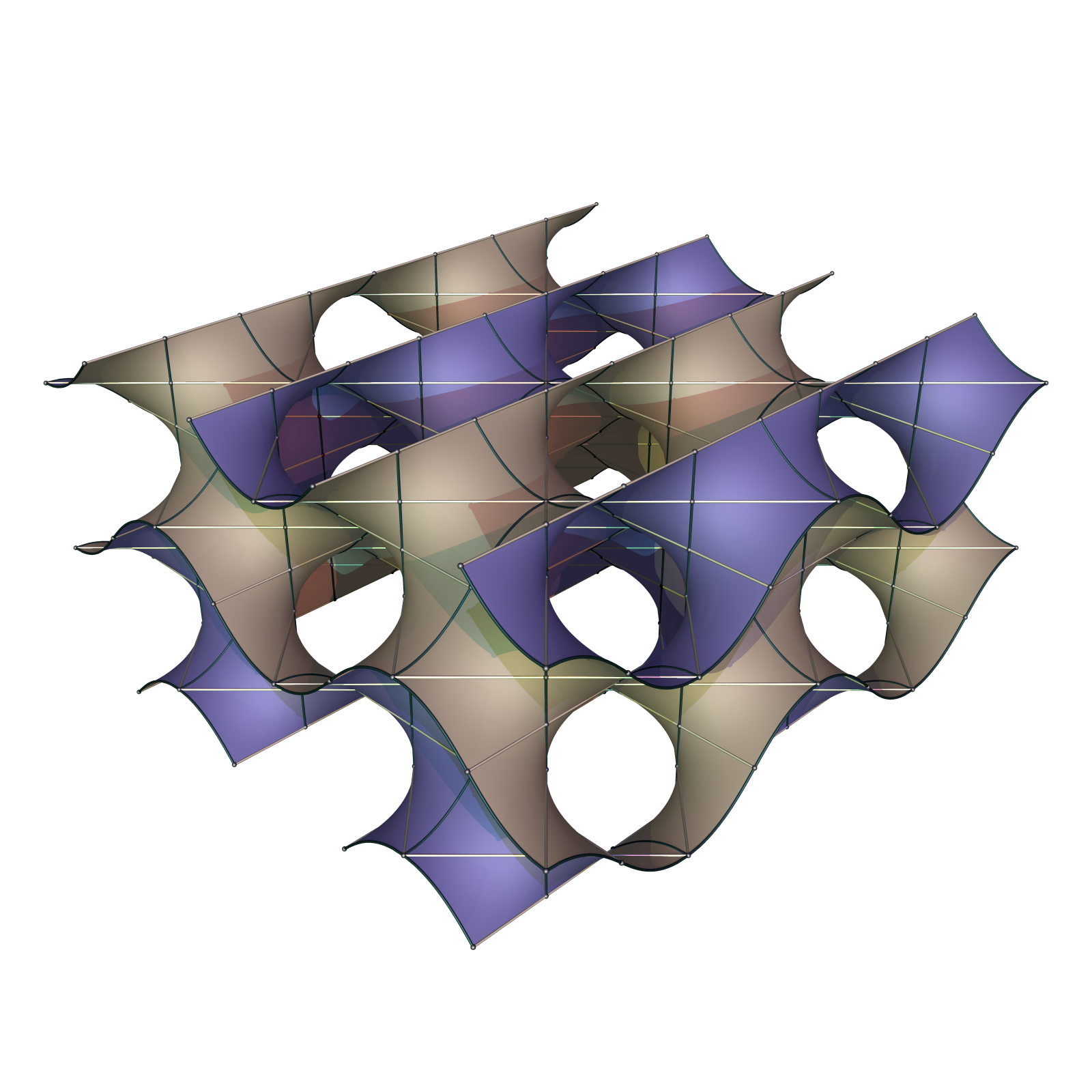}
	\includegraphics[width=0.48\textwidth]{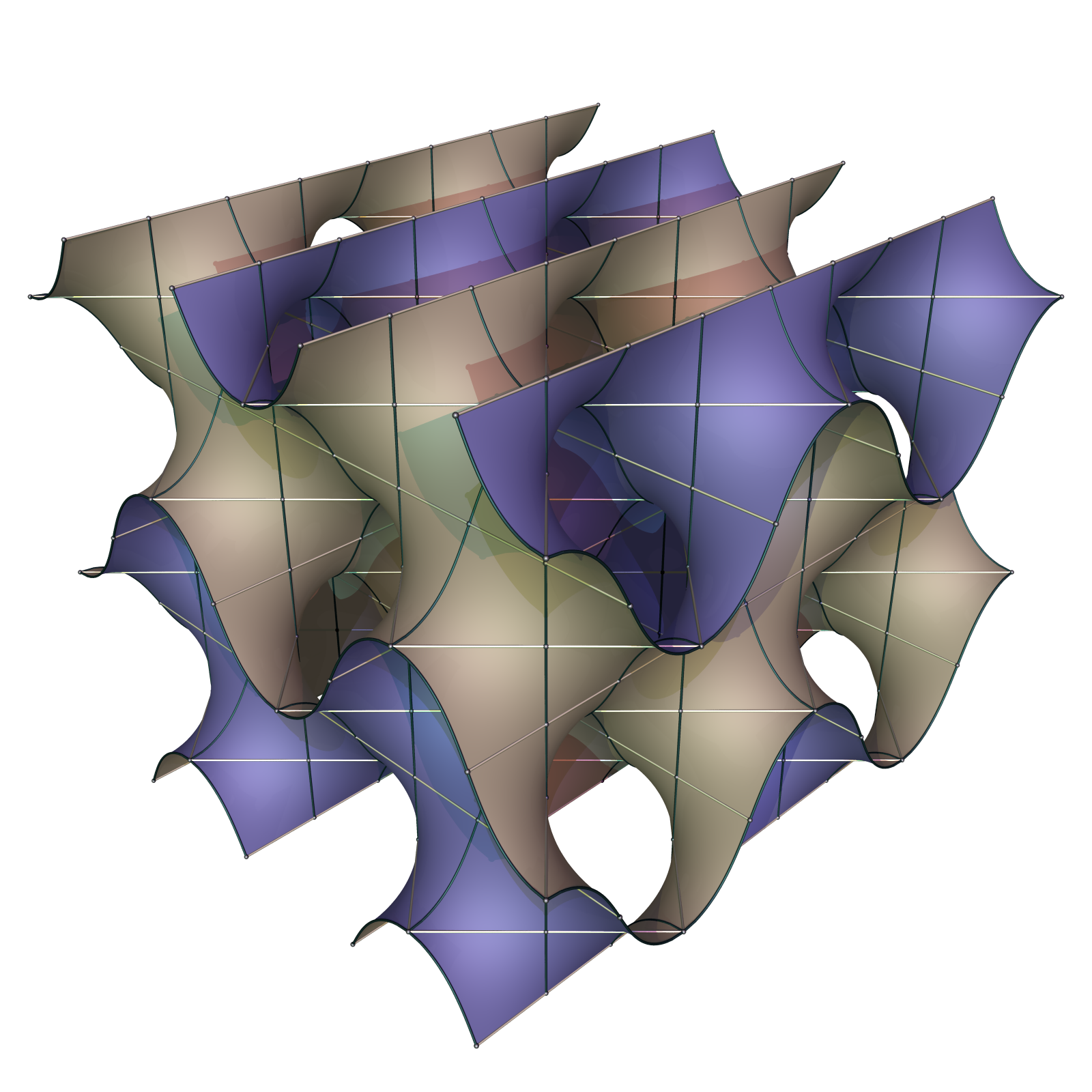}
	\caption{Schwarz' D and the unstable $\tD^*$ surface}
	\label{fig:tD}
\end{figure}

\medskip

The Gauss map of the $\tD^*$ surface has eight branched values at $\pm
\alpha^{\pm 1}$ and $\pm \alpha^{\pm 1} i$, where $\alpha =
\sqrt{(a^*-1)/(a^*+1)}$.  They are the eight roots of $z^8 + k z^4 + 1 = 0$,
where
\[
	k=\alpha^{-4}+\alpha^4=\frac{(a^*-1)^2}{(a^*+1)^2}+\frac{(a^*+1)^2}{(a^*-1)^2} \approx 7.40284
\]
This is precisely the value calculated by Koiso, Piccione and
Shoda~\cite{koiso2014} for a bifurcation instance in the $\tD$ family.  An
explicit bifurcation branch from $\tD^*$ was then missing, but now provided by
the $\tDD$ family.

\begin{remark}
	Surprisingly, numerical computations show that, near the bifurcation point,
	$\tDD$ surfaces have actually {\em smaller} area than the corresponding $\tD$
	surfaces with the same lattice.
\end{remark}

The conjugate of $\tD^*$, denoted by $\tP^*$, was identified
in~\cite{koiso2014} as a bifurcation instance in the $\tP$ family.  We also
find a bifurcation branch from $\tP^*$, denoted by $\tPP$.  As one deforms the
tetragonal lattice, the horizontal handles deform uniformly along the $\tP$
branch.  But along the $\tPP$ branch, the handles in the $x$ direction shrink
while the handles in the $y$ direction expand.  The $\tPP$ family turns out to
be a subfamily of $\oP a$, a 2-parameter orthorhombic deformation family of
Schwarz $P$ surface.  Since $\oP a \subset \cM$, $\tPP$ is less interesting for
understanding non-Meeks surfaces, hence not a focus of the current paper.

\bibliography{References}
\bibliographystyle{alpha}

\end{document}